\documentclass{amsart}
\usepackage{amssymb,amsmath,amsthm,hyperref}
\newtheorem{theorem}{Theorem}[section]
\newtheorem{lemma}[theorem]{Lemma}
\newtheorem{corollary}[theorem]{Corollary}

\newtheorem{proposition}[theorem]{Proposition}

\theoremstyle{definition}

\theoremstyle{remark}
\newtheorem{remark}[theorem]{Remark}

\numberwithin{equation}{section}

\newcommand{\spt}{\mbox{\rm spt}}

\newcommand{\stroke}{\,\mid\,}
\newcommand{\Eis}[2]{\mathcal{E}_{2,#1}(#2)}      
\newcommand{\Ks}[1]{K^{*}(#1)}             

\DeclareSymbolFont{AMSb}{U}{msb}{m}{n}
\DeclareMathSymbol{\Z}{\mathalpha}{AMSb}{"5A}

\DeclareMathSymbol{\nmid}{\mathrel}{AMSb}{"2D}
\DeclareSymbolFont{AMSb}{U}{msb}{m}{n}
\DeclareMathSymbol{\C}{\mathalpha}{AMSb}{"43}
\DeclareMathSymbol{\F}{\mathalpha}{AMSb}{"46}
\DeclareMathSymbol{\N}{\mathalpha}{AMSb}{"4E}
\DeclareMathSymbol{\Q}{\mathalpha}{AMSb}{"51}
\DeclareMathSymbol{\R}{\mathalpha}{AMSb}{"52}
\DeclareMathSymbol{\Z}{\mathalpha}{AMSb}{"5A}

\newcommand{\Hup}{\mathcal{H}} %upper half plane

\allowdisplaybreaks
\begin{document}
%%BEGIN MACROS%%%%%%%%%%%%%%%%%%%%%%%%%%%%%%%%%%%%%%%%%%%%%%%%%%%%%%%%%%%%%%%
\newcommand{\beqs}{\begin{equation*}}
\newcommand{\eeqs}{\end{equation*}}
\newcommand{\beq}{\begin{equation}}
\newcommand{\eeq}{\end{equation}}
\newcommand{\bal}{\begin{align}}
\newcommand{\eal}{\end{align}}
\newcommand{\bals}{\begin{align*}}
\newcommand{\eals}{\end{align*}}
%%%%%%%%%%%%%%%%%%%%%%%%%%%%%%%%%%%%%%%%%%%%%%%%%%%%%%%%%%%%%%%%%%
%%\newcommand\mylabel[1]{\quad\mbox{#1}\label{#1}}
%%MYLABELNO\newcommand\mylabel[1]{\label{#1}}
\newcommand\mylabel[1]{\label{#1}}
%%%%%%%%%%%%%%%%%%%%%%%%%%%%%%%%%%%%%%%%%%%%%%%%%%%%%%%%%%%%%%%%%%
\newcommand\eqn[1]{(\ref{eq:#1})}
\newcommand\thm[1]{\ref{thm:#1}}
\newcommand\lem[1]{\ref{lem:#1}}
\newcommand\propo[1]{\ref{propo:#1}}
\newcommand\cor[1]{\ref{cor:#1}}
\newcommand\sect[1]{\ref{sec:#1}}
\newcommand\leg[2]{\genfrac{(}{)}{}{}{#1}{#2}} %Legendre symbol
\newcommand\ord{\mbox{ord}\,}
%%% special macs
%%END   MACROS%%%%%%%%%%%%%%%%%%%%%%%%%%%%%%%%%%%%%%%%%%%%%%%%%%%%%%%%%%%%%%%
%% > 8*9*5*7*13
%% > ;
%%                                      32760
%% 
\title[Andrews' spt-function modulo $32760$]{
Congruences for Andrews' spt-function \\
modulo $32760$ and extension of Atkin's\\
Hecke-type partition congruences}

%%PRELIMINARY VERSION 

% Information for first author
\author{F.~G.~Garvan}
\address{Department of Mathematics, University of Florida, Gainesville,
Florida 32611-8105}
%%\curraddr{Department of Mathematics and Statistics,
%%Case Western Reserve University, Cleveland, Ohio 43403}
\email{fgarvan@ufl.edu}          
\thanks{The author was supported in part by NSA Grant H98230-09-1-0051.
The first draft of this paper was written October 25, 2010.
%%Most of the research for this paper was done while the author was visiting
%%Mike Hirschhorn at the University of New South Wales in June 2010.
}

%%\author{Frank G. Garvan}
%%\address{Department of Mathematics, University of Florida, Gainesville,
%%Florida 32611-8105}
%%\email{frank@math.ufl.edu}          
%%%%\thanks{Support information for the second author.}

%%%%%%%%%%%%%%%%%%%%%%%%%%%%%%%%%%%%%%%%%%%%%%%%%%%%%%%%%%%%%%%%%%%%%%%%%%%%
%    General info
%%OLD0: \subjclass[2000]{Primary 11P83, 11F11, 11F20, 11F33, 11F37; Secondary
%%05A17, 11P81}

\date{\today}  %%
%%\date{September 20, 2010}  %%
%%\date{October 21, 2010}    %%

%%\dedicatory{Dedicated to the memory of A.O.L. (Oliver) Atkin}
%% (1925-2008)

%% Updated 10-21-10
%%%%%%%%%%%%%%%%%%%%%%%%%%%%%%%%%%%%%%%%%%%%%%%%%%%%%%%%%%%%%%%%%%%%%%%%%%%%

\dedicatory{Dedicated to the memory of A.J. (Alf) van der Poorten, my former 
teacher}
%%\footnote{Alf was one my first year mathematics lecturers at U.N.S.W.
%%in 1973. I was also in his first year honours tutorial class where he
%%introduced us to many interesting things including $\zeta(3)$}
%% (1925-2008)

\begin{abstract}
New congruences are found for Andrews' smallest parts partition function 
$\spt(n)$. The generating function for $\spt(n)$ is related to
the holomorphic part $\alpha(24z)$
of a certain weak Maass form $\mathcal{M}(z)$ of weight $\tfrac{3}{2}$.
We show that a normalized form of the generating function for $\spt(n)$ is 
an eigenform modulo $72$ for the Hecke operators $T(\ell^2)$ for
primes $\ell > 3$, and an eigenform modulo $p$ for $p=5$, $7$ or $13$
provided that $(\ell,6p)=1$. The result for the modulus $3$
was observed earlier by the author and 
considered by Ono and Folsom. Similar congruences for higher
powers of $p$ (namely $5^6$, $7^4$ and $13^2$) occur for the coefficients
of the function $\alpha(z)$. Analogous results for the partition
function were found by Atkin in 1966. Our results depend 
on the recent result of Ono that $\mathcal{M}_{\ell}(z/24)$ is 
a weakly holomorphic modular form of weight $\tfrac{3}{2}$ for the full
modular group where
$$
\mathcal{M}_{\ell}(z) = \mathcal{M}(z) \vert T(\ell^2) - 
\leg{3}{\ell} (1 + \ell) \mathcal{M}(z).
$$ 
\end{abstract}

\maketitle

%section 1%%%%%%%%%%%%%%%%%%%%%%%%%%%%%%%%%%%%%%%%%%%%%%%%%%%%%%%%%%%%%%%%%%%
\section{Introduction} \mylabel{sec:intro}

Andrews \cite{An08b} defined the function $\spt(n)$ as the number 
of  
smallest parts in the partitions of $n$. He related this
function to the second rank moment and proved some surprising congruences
mod $5$, $7$ and $13$. 
Rank and crank moments were introduced by  A.~O.~L.~Atkin and the author
\cite{At-Ga}. 
Bringmann \cite{Br08} studied analytic, asymptotic
and congruence
properties of the generating function for the second rank moment as
a quasi-weak Maass form. 
Further congruence properties of Andrews'
spt-function were found by the author \cite{Ga10a}, \cite{Ga10d}, Folsom and Ono \cite{Fo-On}
and Ono \cite{On10}.
In particular, Ono \cite{On10}
proved that if $\leg{1-24n}{\ell}=1$ then
\beq
\spt(\ell^2 n - \tfrac{1}{24}(\ell^2-1)) \equiv 0 \pmod{\ell},
\mylabel{eq:sptellcong}
\eeq
for any prime $\ell \ge 5$. This amazing result was originally conjectured
by the author\footnote{The congruence \eqn{sptellcong} was first
conjectured by the author in a Colloquium given at the University of
Newcastle, Australia on July 17, 2008.}.
Earlier special cases were observed by Tina Garrett \cite{Ga-PC2007}
and her students. Recently the author \cite{Ga10d} has proved the
following congruences for powers of $5$, $7$ and $13$.
For $a$, $b$, $c\ge3$,
\begin{align}
\spt(5^a n + \delta_a) + 5\, \spt(5^{a-2} n + \delta_{a-2}) &\equiv 0  \pmod{5^{2a-3}},
\mylabel{eq:spt5acong}\\
\spt(7^b n + \lambda_b) + 7\, \spt(7^{b-2} n + \lambda_{b-2}) &\equiv 0  \pmod{7^{\lfloor\frac{1}{2}(3b-2)\rfloor}},
\mylabel{eq:spt7bcong}\\
\spt(13^c n + \gamma_c) - 13\, \spt(13^{c-2} n + \gamma_{c-2}) &\equiv 0  \pmod{13^{c-1}},   
\mylabel{eq:spt13ccong}
\end{align}
where $\delta_a$, $\lambda_b$ and $\gamma_c$ are the 
least nonnegative residues of the reciprocals of $24$ mod
$5^a$, $7^b$ and $13^c$ respectively.  

As in \cite{On10}, \cite{Ga10d} we define
\beq
 \mathbf{a}(n) := 12 \spt(n)  + (24n -1) p(n),
\mylabel{eq:adef}
\eeq
for $n\ge 0$, and define
\beq
\alpha(z) := \sum_{n\ge0}  \mathbf{a}(n) q^{n - \tfrac{1}{24}},
\mylabel{eq:alphadef}
\eeq
where as usual $q = \exp(2\pi i z)$ and $\Im(z) > 0$.
We note that $\spt(0)=0$ and $p(0)=1$.
Bringmann \cite{Br08} showed that $\alpha(24z)$ is the
holomorphic part of the weight $\tfrac{3}{2}$ weak Maass form $\mathcal{M}(z)$
on $\Gamma_0(576)$ with Nebentypus
$\chi_{12}$ where
\beq
\mathcal{M}(z) := \alpha(24z) - \frac{3i}{\pi\sqrt{2}} \,   
\int_{-\overline{z}}^{i\infty} \frac{\eta(24\tau) \, d\tau}{(-i(\tau+z))^{\tfrac32}},
\mylabel{eq:Mdef}
\eeq
$\eta(z) := q^{\tfrac{1}{24}}\prod_{n=1}^\infty(1-q^n)$ is the Dedekind
eta-function,  the function $\alpha(z)$ is defined in \eqn{alphadef},
and 
\beq
\chi_{12}(n) = 
\begin{cases}
1 & \mbox{if $n\equiv\pm1\pmod{12}$,}\\
-1 & \mbox{if $n\equiv\pm5\pmod{12}$,}\\
0 & \mbox{otherwise.}
\end{cases}
\mylabel{eq:chi12}
\eeq
Ono \cite{On10} showed that for $\ell\ge5$ prime, the operator
\beq
T(\ell^2) - \chi_{12}(\ell) {\ell} (1 + \ell)
\eeq
annihilates the nonholomorphic part of $\mathcal{M}(z)$, and    
the function 
$\mathcal{M}_{\ell}(z/24)$ is
a weakly holomorphic modular form of weight $\tfrac{3}{2}$ for the full
modular group where
\beq
\mathcal{M}_{\ell}(z) 
= \mathcal{M}(z) \vert T(\ell^2) - \chi_{12}(\ell) (1 + \ell) \mathcal{M}(z)
= \alpha(24z) \vert T(\ell^2) - \chi_{12}(\ell) (1 + \ell) \alpha(24z).
\mylabel{eq:Melldef}
\eeq
In fact he obtained                           
\begin{theorem}[Ono \cite{On10}]
\mylabel{thm:Othm}
If  $\ell\ge 5$ is prime then the function
\beq
\mathcal{M}_{\ell}(z/24) \, \eta(z)^{\ell^2}
\mylabel{eq:Melleta}
\eeq
is an entire modular form of weight $\tfrac{1}{2}(\ell^2+3)$ for
the full modular group $\Gamma(1)$.
\end{theorem}
Applying this theorem Ono obtained
\beq
\mathcal{M}_{\ell}(z) \equiv 0 \pmod{\ell}.
\mylabel{eq:Mellcong}
\eeq
The congruence \eqn{sptellcong} then follows easily.

Folsom and Ono \cite{Fo-On} sketched the proof of the following
\begin{theorem}[Folsom and Ono]
\mylabel{thm:FOthm}
If  $\ell\ge 5$ is prime then
\beq
\spt(\ell^2 n - s_\ell)
+ \chi_{12}(\ell) \leg{1-24n}{\ell} \spt(n)
+ \ell \,\spt\left( \frac{n + s_\ell}{\ell^2} \right)
\equiv
 \chi_{12}(\ell) \, (1 + \ell)\, \spt(n) \pmod{3},
\mylabel{eq:sptmod3}
\eeq
where
\beq
s_{\ell} = \frac{1}{24}(\ell^2 - 1).
\mylabel{eq:selldef}
\eeq
\end{theorem}
This result was observed earlier by the author. In this paper we prove
a much stronger result.
\begin{theorem}
\mylabel{thm:mainthm}
\begin{enumerate}
\item[(i)]
If $\ell\ge 5$ is prime then 
\beq
\spt(\ell^2 n - s_\ell)
+ \chi_{12}(\ell) \leg{1-24n}{\ell} \spt(n)
+ \ell \,\spt\left( \frac{n + s_\ell}{\ell^2} \right)
\equiv
 \chi_{12}(\ell) \, (1 + \ell) \, \spt(n) \pmod{72}.
\mylabel{eq:sptmod72}
\eeq
\item[(ii)]
If $\ell\ge 5$ is prime, $t=5$, $7$ or $13$ and $\ell\ne t$ then
\beq
\spt(\ell^2 n - s_\ell)
+ \chi_{12}(\ell) \leg{1-24n}{\ell} \spt(n)
+ \ell \,\spt\left( \frac{n + s_\ell}{\ell^2} \right)
\equiv
 \chi_{12}(\ell) \, (1 + \ell) \, \spt(n) \pmod{t}.
\mylabel{eq:sptmodt}
\eeq
\end{enumerate}
\end{theorem}
Of course this implies the
\begin{corollary}
\mylabel{cor:Jell}
If $\ell$ is prime and $\ell\not\in\{2,3,5,7,13\}$ then          
\beq
\spt(\ell^2 n - s_\ell)
+ \chi_{12}(\ell) \leg{1-24n}{\ell} \spt(n)
+ \ell \,\spt\left( \frac{n + s_\ell}{\ell^2} \right)
\equiv
 \chi_{12}(\ell) \, (1 + \ell) \, \spt(n) \pmod{32760}.
\mylabel{eq:sptmod}
\eeq
\end{corollary}
This congruence modulo $32760=2^3\cdot 3^2 \cdot 5 \cdot 7 \cdot 13$
is the congruence referred in the title of this paper.
%%
%% > 8*9*5*7*13
%% > ;
%%                                      32760
%%

In 1966, Atkin \cite{At68b} found a similar congruence for the partition
function.
\begin{theorem}[Atkin]
\mylabel{thm:Atkin}
Let $t=5$, $7$, or $13$, and $c=6$, $4$, or $2$ respectively.
Suppose $\ell \ge 5$ is prime and $\ell\ne t$. If $\leg{1-24n}{t}=-1$, then
\beq
\ell^3 \, p(\ell^2 n - s_\ell)
+ \ell \chi_{12}(\ell) \leg{1-24n}{\ell} p(n)
+  \,p\left( \frac{n + s_\ell}{\ell^2} \right)
\equiv
 \gamma_t \, p(n) \pmod{t^c},
\mylabel{eq:ptnmod}
\eeq
where $\gamma_t$ is an integral constant independent of $n$.
\end{theorem}
We find that there is a corresponding result for the function $\mathbf{a}(n)$
defined in \eqn{adef}.
\begin{theorem}
\mylabel{thm:aAtkincong}
Let $t=5$, $7$, or $13$, and $c=6$, $4$, or $2$ respectively.
Suppose $\ell \ge 5$ is prime and $\ell\ne t$.
If $\leg{1-24n}{t}=-1$, then
\beq
\mathbf{a}(\ell^2 n - s_\ell)
+ \chi_{12}(\ell) \leg{1-24n}{\ell} \mathbf{a}(n)
+ \ell \,\mathbf{a}\left( \frac{n + s_\ell}{\ell^2} \right)
\equiv
 \chi_{12}(\ell) \, (1 + \ell) \, \mathbf{a}(n) \pmod{t^c}.
\mylabel{eq:amod}
\eeq
\end{theorem}

In Section \sect{proofmainthm} we prove Theorem \thm{mainthm}.
The method involves reviewing the action of weight $-\tfrac{1}{2}$ 
Hecke operators $T(\ell^2)$
on the function $\eta(z)^{-1}$ and 
doing a careful study of the action of weight $\tfrac{3}{2}$ 
Hecke operators on the function $\frac{d}{dz} \eta(z)^{-1}$ modulo 
$5$, $7$, $13$, $27$ and $32$. In Section \sect{proofthmaAtkincong}
we prove Theorem \thm{aAtkincong}. The method involves extending
Atkin's \cite{At68b} on modular functions to weight two
modular forms on $\Gamma_0(t)$ for $t=5$, $7$ and $13$.
The proof of both Theorems \thm{mainthm} and \thm{aAtkincong}
depend on Ono's Theorem \thm{Othm}.

%section 2%%%%%%%%%%%%%%%%%%%%%%%%%%%%%%%%%%%%%%%%%%%%%%%%%%%%%%%%%%%%%%%%%%%
\section{Proof of Theorem \thm{mainthm}} \mylabel{sec:proofmainthm}

In this section we prove Theorem \thm{mainthm}.
Atkin \cite{At68} showed essentially that applying certain weight
$-\tfrac{1}{2}$ Hecke operators $T(\ell^2)$ to the function $\eta(z)^{-1}$    
produces a function with the same multiplier system as $\eta(z)^{-1}$
and thus $\eta(z)$ times this function is a certain polynomial 
(depending on $\ell$)
of Klein's modular invariant $j(z)$. 
We review Ono's \cite{On10b}  recent explicit
form for these polynomials. Although our proof does not depend
on Ono's result it is quite useful for computational purposes.
The action of the corresponding weight $\tfrac{3}{2}$ Hecke
operators on  $\frac{d}{dz} \eta(z)^{-1}$ can be given in terms of the
same polynomials. See Theorem \thm{dHecke} below. To finish the proof
of the theorem we need to make a careful
study of the action of these operators modulo 
$5$, $7$, $13$, $27$ and $32$.

For $\ell\ge5$ prime we define
\beq
Z_{\ell}(z) = \sum_{n=-s_\ell}^\infty \left(\ell^3 \, p(\ell^2 n - s_\ell)
+ \ell \chi_{12}(\ell) \leg{1-24n}{\ell} p(n)
+  \,p\left( \frac{n + s_\ell}{\ell^2} \right) \right) q^{n-\frac{1}{24}}.
\mylabel{eq:Zell}
\eeq
%%\begin{proposition}[Atkin \cite[Lemma 1]{At68b}]
\begin{proposition}[Atkin \cite{At68b}]
\mylabel{propo:AtkinLemma1}
The function $Z_\ell(z) \, \eta(z)$ is a modular function
on the full modular group $\Gamma(1)$.
\end{proposition}
It follows that $Z_\ell(z) \, \eta(z)$ is a polynomial in $j(z)$,
where $j(z)$ is Klein's modular invariant
\beq
j(z) := \frac{E_4(z)^2}{\Delta(z)} = q^{-1} + 744 + 196884 q + \cdots,
\mylabel{eq:jdef}
\eeq
$E_2(z)$, $E_4(z)$, $E_6(z)$ are the usual Eisenstein series
\beq
E_2(z) := 1 - 24 \sum_{n=1}^\infty \sigma_1(n) q^n, \qquad 
E_4(z) := 1 + 240 \sum_{n=1}^\infty \sigma_3(n) q^n, \qquad 
E_6(z) := 1 - 504 \sum_{n=1}^\infty \sigma_5(n) q^n, 
\mylabel{eq:E246def}
\eeq
$\sigma_k(n) = \sum_{d\mid n} q^k$, and $\Delta(z)$ is Ramanujan's
function
\beq
\Delta(z) := \eta(z)^{24} = q \prod_{n=1}^\infty (1 - q^n)^{24}.
\mylabel{eq:Deltadef}
\eeq
In a recent paper, Ono \cite{On10b} has found a nice formula
for this polynomial. We define
\beq
E(q) := \prod_{n=1}^\infty (1 - q^n) = q^{-\frac{1}{24}} \eta(z),
\mylabel{eq:Eqdef}
\eeq
and a sequence of polynomials $A_m(x)\in\Z[x]$ by
\begin{align}
\sum_{m=0}^\infty A_m(x) q^m &= E(q) \, \frac{E_4(z)^2 E_6(z)}{\Delta(z)}
                                    \, \frac{1}{j(z) - x}
\mylabel{eq:Amxdef}\\
&= 1 + (x-745) q + (x^2 - 1489 x + 160511) q^2 + \cdots.
\nonumber
\end{align}
\begin{theorem}[Ono \cite{On10b}]
\mylabel{thm:OnoHeckeThm}
For $\ell\ge5$ prime
\beq
Z_{\ell}(z) \, \eta(z) = \ell \, \chi_{12}(\ell) + A_{s_\ell}(j(z)),
\mylabel{eq:OnoPolys}
\eeq
where $Z_\ell(z)$ is given in \eqn{Zell}, and $s_\ell$ is given 
in \eqn{selldef}.
\end{theorem}

We define a sequence of polynomials $C_\ell(x)\in\Z[x]$ by
\begin{align}
C_\ell(x) & := \ell \, \chi_{12}(\ell) + A_{s_\ell}(x),
\mylabel{eq:Celldef}\\
&= \sum_{n=0}^{s_\ell} c_{n,\ell} x^n,
\nonumber
\end{align}
so that
\beq
Z_{\ell}(z) \, \eta(z) = C_\ell(j(z)).
\mylabel{eq:ZellC}
\eeq
We define
\beq
d(n) := (24n -1)\,p(n),
\mylabel{eq:ddef}
\eeq
%%
%% 1/eta(24z) = sum{n>=0} p(n) q^(24n-1)
%% sum d(n) q^(24n-1) = q d/dq 1/eta(24z) = .. 
%%
so
that
\beq
\sum_{n=0}^\infty d(n) q^{24n-1} = q \frac{d}{dq} \frac{1}{\eta(24z)}
= - \frac{E_2(24z)}{\eta(24z)},
\mylabel{eq:dgen}
\eeq
and
%%\mathbf{a}(n) := 12 \spt(n)  + (24n -1) p(n),
\beq
\mathbf{a}(n) = 12 \spt(n)  + d(n).
\mylabel{eq:asd}
\eeq
For $\ell\ge5$ prime we define
\beq
\Xi_{\ell}(z) = \sum_{n=-s_\ell}^\infty \left(d(\ell^2 n - s_\ell)
+ \chi_{12}(\ell) \left(\leg{1-24n}{\ell}-1-\ell\right) d(n)
+  \ell\,d\left( \frac{n + s_\ell}{\ell^2} \right) \right) q^{n-\frac{1}{24}}.
\mylabel{eq:Xidef}
\eeq
We then
have the following analogue of Theorem \thm{OnoHeckeThm}.
\begin{theorem}
\mylabel{thm:dHecke}
For $\ell\ge5$ prime we have
\begin{align}    
\ell \, \Xi_\ell(z) \, \eta(z) \, \Delta(z)^{s_\ell} 
&=  -\sum_{n=0}^{s_\ell} c_{n,\ell} \, E_4(z)^{3n-1} 
               \, \Delta(z)^{s_\ell - n}
               \left( 24n E_6(z) + E_4(z) E_2(z) \right)
\mylabel{eq:Xiid}\\
& \qquad + \chi_{12}(\ell) \ell (1 + \ell) E_2(z) \, \Delta(z)^{s_\ell},
\nonumber
\end{align}
where the coefficients $ c_{n,\ell}$ are defined by \eqn{Amxdef} and \eqn{Celldef}.
\end{theorem}
\begin{proof}
Suppose $\ell\ge5$ is prime. In equation \eqn{ZellC} we replace
$z$ by $24z$, apply the operator $q\frac{d}{dq}$ and
replace $z$ by $\tfrac{1}{24}z$ to obtain
\beq
\ell \, \Xi_\ell(z) \, \eta(z) 
= 24 C_\ell'(j(z)) \, q\frac{d}{dq} (j(z)) + 
(\chi_{12}(\ell) \ell (1 + \ell) - C_{\ell}(j(z)) \, E_2(z) 
\mylabel{eq:Xiid1}
\eeq
The result then follows easily from the identities
\beq
j(z) \, \Delta(z) = E_4(z)^3, \qquad
 q\frac{d}{dq} (\Delta(z)) = \Delta(z) \, E_2(z),\quad  \mbox{and}\qquad
 q\frac{d}{dq} (j(z)) \, \Delta(z) = - E_4(z)^2 E_6(z),
\mylabel{eq:jids}
\eeq
which we leave as an easy exercise.
\end{proof}

We are now ready to prove Theorem \thm{mainthm}.
A standard calculation gives the following congruences.
\beq
E_4(z)^3 - 720 \, \Delta(z) \equiv 1 \pmod{65520},\quad\mbox{and}\qquad
E_2(z) \equiv E_4(z)^2 E_6(z) 
\pmod{65520}.      
\mylabel{eq:E246mod65520}
\eeq

We now use \eqn{E246mod65520} to reduce \eqn{Xiid1} modulo
$65520$.
\begin{align}
&\ell \, \Xi_\ell(z) \, \eta(z) \, \Delta(z)^{s_\ell} 
\mylabel{eq:Xiidcong1}\\
& \equiv  -\sum_{n=0}^{s_\ell} c_{n,\ell} \, E_4(z)^{3n-1} 
               \, \Delta(z)^{s_\ell - n}
         \left( 24n E_6(z)(E_4(z)^3 - 720 \, \Delta(z)) + E_4(z)^3 E_6(z) \right)
\nonumber\\
& \qquad\qquad + \chi_{12}(\ell) \ell(1 + \ell) \, E_4(z)^2 E_6(z)  
                                                \, \Delta(z)^{s_\ell}
\pmod{65520}      
\nonumber\\
& \equiv 
-\sum_{n=0}^{s_\ell} (24n+1) c_{n,\ell} \, E_4(z)^{3n+2} \, E_6(z)
               \, \Delta(z)^{s_\ell - n}
\nonumber\\
&\qquad +\sum_{n=0}^{s_\ell} 720 \cdot 24n c_{n,\ell} \, E_4(z)^{3n-1} \, E_6(z)
               \, \Delta(z)^{s_\ell - n+1}
+ \chi_{12}(\ell) \ell(1 + \ell) \,  E_4(z)^2 E_6(z) \, \Delta(z)^{s_\ell}
\pmod{65520}      
\nonumber\\
& \equiv 
\left(720\, c_{1,\ell} -  c_{0,\ell} + \chi_{12}(\ell) \ell(1 + \ell)\right)
 \, E_4(z)^{2} \, E_6(z) \, \Delta(z)^{s_\ell}
\nonumber\\
&\qquad +
\sum_{n=1}^{s_\ell-1} \left(720\cdot 24(n+1) c_{n+1,\ell} - (24n+1) c_{n,\ell}\right)
             \, E_4(z)^{3n+2} \, E_6(z) \, \Delta(z)^{s_\ell - n}
\nonumber\\
& \qquad - (24s_\ell+1) c_{s_\ell}  E_4(z)^{3s_\ell+2} \, E_6(z)
\pmod{65520}.
\nonumber 
\end{align}
We define
\beq
\mathcal{A}_{\ell}(z) := \sum_{n=-s_\ell}^\infty \left(\mathbf{a}(\ell^2 n - s_\ell)
+ \chi_{12}(\ell)  \left(\leg{1-24n}{\ell}-1-\ell\right)\mathbf{a}(n)
+  \ell\,\mathbf{a}\left( \frac{n + s_\ell}{\ell^2} \right) \right) q^{n-\frac{1}{24}}
\mylabel{eq:Aelldef}
\eeq
and
\beq
\mathcal{S}_{\ell}(z) := \sum_{n=1}^\infty \left(\spt(\ell^2 n - s_\ell)
+ \chi_{12}(\ell)  \left(\leg{1-24n}{\ell}-1-\ell\right)\spt(n)
+  \ell\,\spt\left( \frac{n + s_\ell}{\ell^2} \right) \right) q^{n-\frac{1}{24}},
\mylabel{eq:Selldef}
\eeq
so that
\beq
\mathcal{A}_{\ell}(z) = 12 \, \mathcal{S}_{\ell}(z) 
+ \Xi_{\ell}(z) = \mathcal{M}_\ell(z/24).
\mylabel{eq:ASXi}                          
\eeq
By Theorem \thm{Othm} and equation \eqn{Melldef} we see that the function
\beq
\ell \mathcal{A}_{\ell}(z) \, \eta(z) \, \Delta(z)^{s_\ell} \in 
M_{\tfrac{1}{2}(\ell^2+3)}(\Gamma(1)),
\eeq
the space of entire modular forms of weight $\tfrac{1}{2}(\ell^2+3)$ on
$\Gamma(1)$. Since $\tfrac{1}{2}(\ell^2+3)= 2 + 12 s_\ell$ the set
\beq
\{
E_4(z)^{3n-1} \, E_6(z) \, \Delta(z)^{s_\ell - n} \,:\, 1 \le n \le s_\ell
\}
\eeq
is a basis. Hence there are integers $b_{n,\ell}$ ($ 1 \le n \le s_\ell$)
such that
\beq
\mathcal{A}_{\ell}(z) \, \eta(z) \, \Delta(z)^{s_\ell} 
= \sum_{n=1}^{s_\ell} b_{n,\ell} E_4(z)^{3n-1} \, E_6(z) \, \Delta(z)^{s_\ell - n}.
\mylabel{eq:Aellid}
\eeq
Using \eqn{E246mod65520} we find that
\begin{align}
\mathcal{A}_{\ell}(z) \, \eta(z) \, \Delta(z)^{s_\ell} 
&\equiv 
 - 720 b_{1,\ell} \, E_4(z)^{2} \, E_6(z) \, \Delta(z)^{s_\ell}
\mylabel{eq:Aellidmod65520}\\
&\qquad +
\sum_{n=1}^{s_\ell-1} (b_{n,\ell} - 720 b_{n+1,\ell}) 
             \, E_4(z)^{3n+2} \,  E_6(z) \, \Delta(z)^{s_\ell - n}
\nonumber\\
& \qquad + b_{s_\ell,\ell}  E_4(z)^{3s_\ell+2} \, E_6(z)
\pmod{65520}.
\nonumber 
\end{align}
%%NOTE: In the original version of (Aellid) I had 
%%      ell*A[ell](z) ....
%%      It is better to remove the constant ell since
%%      this is more useful later.
%%
By \eqn{Xiidcong1}, \eqn{ASXi} and \eqn{Aellid} we deduce that
there are integers $a_{n,\ell}$ ($0 \le n \le s_\ell$) such that
\beq
12 \, \ell \, \mathcal{S}_{\ell}(z) \, \eta(z) \, \Delta(z)^{s_\ell} 
\equiv 
\sum_{n=0}^{s_\ell} 
  a_{n,\ell} \, E_4(z)^{3n+2} \, E_6(z) \, \Delta(z)^{s_\ell - n}
\pmod{65520}.
\eeq
It follows that
\beq
12 \, \ell \, \mathcal{S}_{\ell}(z) \equiv 0 \pmod{65520},
\mylabel{eq:nicecong1}
\eeq
since
\begin{align}
 \ord_{i\infty}\left(12 \, \ell \, \mathcal{S}_{\ell}(z) \, \eta(z) \, \Delta(z)^{s_\ell}\right) &= s_{\ell} + 1,
\mylabel{eq:ords}\\
0 \le \ord_{i\infty}\left(E_4(z)^{3n+2} \, E_6(z) \, \Delta(z)^{s_\ell - n}\right)
   &\le s_\ell,
\nonumber\\
  E_4(z)^{3n+2} \, E_6(z) \, \Delta(z)^{s_\ell - n} &= q^{s_\ell -n} + \cdots,
\nonumber
\end{align}
for $0 \le n \le s_\ell$ and all functions have integral coefficients.
%% ifactor(65520);
%%                               4    2             
%%                            (2)  (3)  (5) (7) (13)
%% 
Since $65550 = 2^4 \cdot 3^2 \cdot 5 \cdot 7 \cdot 13$, the congruence
\eqn{nicecong1} implies Part (ii) of Theorem \thm{mainthm}.
To prove Part (i) we need to work a little harder.
We note that the congruence \eqn{nicecong1} does imply   
\beq
\mathcal{S}_{\ell}(z) \equiv 0 \pmod{12}.
\mylabel{eq:nicecong1a}
\eeq
We need to show this congruence actually holds modulo $72$.

        First we show the congruence holds modulo $8$ by studying
$\Xi_\ell(z)$ modulo $32$. We need the congruences,
\beq
E_2(z) \equiv E_4(z) \, E_6(z) + 16 \Delta(z) \pmod{32},\quad\mbox{and}\qquad
E_4(z)^2 \equiv 1 \pmod{32},                             
\mylabel{eq:E24mod32}
\eeq
which are routine to prove. We proceed as in the proof of \eqn{Xiidcong1} to
find that
\begin{align}
&\ell \, \Xi_\ell(z) \, \eta(z) \, \Delta(z)^{s_\ell} 
\mylabel{eq:Xiidcong2}\\
& \equiv 
\left(\chi_{12}(\ell) \ell(1 + \ell)  - c_{0,\ell} - 16 c_{1,\ell}\right)
 \, E_2(z)\, \Delta(z)^{s_\ell}
\nonumber\\
&\qquad -
\sum_{n=1}^{s_\ell-1} \left((24n+1) c_{n,\ell} + 16  c_{n+1,\ell}\right)
             \, E_4(z)^{3n-1} \, E_6(z) \, \Delta(z)^{s_\ell - n}
\nonumber\\
& \qquad - (24s_\ell+1) c_{s_\ell}  E_4(z)^{3s_\ell-1} \, E_6(z)
\pmod{32}.
\nonumber 
\end{align}
By \eqn{Xiidcong2}, \eqn{ASXi} and \eqn{Aellid} we deduce that
there are integers $a_{n,\ell}'$ ($0 \le n \le s_\ell$) such that
\beq
12 \, \ell \,  \mathcal{S}_{\ell}(z) \,  \eta(z) \,  \Delta(z)^{s_\ell} 
\equiv 
\sum_{n=1}^{s_\ell} 
  a_{n,\ell}' \,  E_4(z)^{3n-1} \,  E_6(z) \,  \Delta(z)^{s_\ell - n}
+ a_{0,\ell}' \,  E_2(z) \,  \Delta(z)^{s_\ell}
\pmod{32}.
\eeq
Arguing as before,
it follows that
\beq
12\mathcal{S}_{\ell}(z) \equiv 0 \pmod{32},\quad\mbox{and}\qquad
\mathcal{S}_{\ell}(z) \equiv 0 \pmod{8}.
\mylabel{eq:nicecong2}
\eeq

To complete the proof, we need to study $\Xi_{\ell}(z)$ modulo $27$.
We need the congruences,
\beq
E_2(z) \equiv E_4(z)^5  + 18 \Delta(z) \pmod{27}\quad\mbox{and}\qquad
E_6(z) \equiv E_4(z)^6 \pmod{27},                             
\mylabel{eq:E246mod27}
\eeq
which are routine to prove. We proceed as in the proof of \eqn{Xiidcong1} and
\eqn{Xiidcong2} to find
that
\begin{align}
&\ell \,  \Xi_\ell(z) \,  \eta(z) \,  \Delta(z)^{s_\ell} 
\mylabel{eq:Xiidcong3}\\
& \equiv 
\left(\chi_{12}(\ell) \ell(1 + \ell)  - c_{0,\ell} - 18 c_{1,\ell}\right)
 \,  E_2(z)\,  \Delta(z)^{s_\ell}
\nonumber\\
&\qquad -
\sum_{n=1}^{s_\ell-1} \left((24n+1) c_{n,\ell} + 18  c_{n+1,\ell}\right)
             \,  E_4(z)^{3n-1} \,  E_6(z) \,  \Delta(z)^{s_\ell - n}
\nonumber\\
& \qquad - (24s_\ell+1) c_{s_\ell}  E_4(z)^{3s_\ell-1} \,  E_6(z)
\pmod{27}.
\nonumber 
\end{align}
By \eqn{Xiidcong3}, \eqn{ASXi} and \eqn{Aellid} we deduce that
there are integers $a_{n,\ell}''$ ($0 \le n \le s_\ell$) such that
\beq
12 \,  \ell \,  \mathcal{S}_{\ell}(z) \,  \eta(z) \,  \Delta(z)^{s_\ell} 
\equiv 
\sum_{n=1}^{s_\ell} 
  a_{n,\ell}''\,  E_4(z)^{3n-1} \,  E_6(z) \,  \Delta(z)^{s_\ell - n}
+ a_{0,\ell}'' \,  E_2(z) \,  \Delta(z)^{s_\ell}
\pmod{27}.
\eeq
Arguing as before,
it follows that
\beq
12\mathcal{S}_{\ell}(z) \equiv 0 \pmod{27},\quad\mbox{and}\qquad
\mathcal{S}_{\ell}(z) \equiv 0 \pmod{9}.
\mylabel{eq:nicecong3}
\eeq
The congruences \eqn{nicecong2} and \eqn{nicecong3} give
\eqn{sptmod72} and this completes the proof of Theorem \thm{mainthm}.

%section 3%%%%%%%%%%%%%%%%%%%%%%%%%%%%%%%%%%%%%%%%%%%%%%%%%%%%%%%%%%%%%%%%%%%
\section{Proof of Theorem \thm{aAtkincong}} \mylabel{sec:proofthmaAtkincong}

In this section we prove Theorem \thm{aAtkincong}. 
Atkin \cite{At68b} proved Theorem \thm{Atkin} by constructing
certain special modular functions on  $\Gamma_0(t)$ and $\Gamma_0(t^2)$
for $t=5$, $7$ and $13$. We attack the problem by extending
Atkin's results to the corresponding weight $2$ case.

Let $GL_2^{+}(\R)$ denote the group of all real $2\times2$ matrices with
positive determinant. $GL_2^{+}(\R)$ acts on the complex upper half plane
$\mathcal{H}$ by linear fractional transformations. 
We define the slash operator for modular forms of
integer weight. Let $k\in\Z$.
For a function $f\,:\,\Hup\longrightarrow\C$
and $L=\begin{pmatrix} a & b \\ c & d \end{pmatrix}\in GL_2^{+}(\R)$
we define
\beq
f(z)\,\mid_{k}\,L = 
f\,\mid_{k}\,L = f\,\mid\,L  =
(\det L)^{\tfrac{k}{2}} (cz + d)^{-k} f(L z).
\mylabel{eq:slashdef}
\eeq
Let $\Gamma'\subset \Gamma(1)$ (a subgroup of finite index).  We 
say $f(z)$ is a \textit{weakly holomorphic modular form} of weight $k$ on 
$\Gamma'$
if $f(z)$ is holomorphic on the upper half plane $\mathcal{H}$,
$f(z)\,\mid_{k}\,L = f(z)$ for all $L$ in $\Gamma'$, and $f(z)$ has
at most polar singularities in the local variables at the cusps
of the fundamental region of $\Gamma'$.
We say $f(z)$ is a \textit{weakly holomorphic modular function} if it
is a weakly holomorphic modular form of weight $0$.
We say $f(z)$ is an \textit{entire modular form} of weight $k$ on $\Gamma'$
if it is  a \textit{weakly holomorphic modular form} that is holomorphic
at the cusps of the fundamental region of $\Gamma'$.
We denote the space of entire modular forms of weight $k$ on
$\Gamma'$ by $M_k(\Gamma')$.

Suppose that $t\ge 5$ is prime. We need 
\begin{gather*}
W_t = W = \begin{pmatrix} 0 & -1 \\ t & 0 \end{pmatrix},\quad
R = \begin{pmatrix} 1 & 0 \\ -1 & 1 \end{pmatrix},\quad
V_a = \begin{pmatrix} a & \lambda \\ t & a' \end{pmatrix},\quad
B_t = \begin{pmatrix} t & 0 \\ 0 & 1 \end{pmatrix},\quad\\
T_{b,t} = \begin{pmatrix} 1 & b \\ 0 & 1 \end{pmatrix},\quad
Q_{b,t} = \begin{pmatrix} 1/t & b/t \\ 0 & 1 \end{pmatrix},
\end{gather*}
where for $1 \le a \le t-1$, $a'$ is uniquely defined by
$1 \le a' \le t-1$, and $a'a - \lambda t=1$.        
We have
\begin{align}
  B_t \, R^{at} &= W_t \, V_a \, T_{-a'/t}
\mylabel{eq:matrels1}\\
  R^{at} \, W_t &= W_{t^2} \, Q_{a,t}.
\mylabel{eq:matrels2}
\end{align}

We define
\beq
\Phi_t(z) = \Phi(z) := \frac{\eta(z)}{\eta(t^2z)}.
\mylabel{eq:Phidef}
\eeq
Then $\Phi_t(z)$ is a modular function of $\Gamma_0(t)$,
\beq
\Phi_t(z) \,\mid\, W_{t^2} =  t\,\Phi_t(z)^{-1}
\qquad\mbox{(\cite[(24)]{At68b})},
\mylabel{eq:PhitransW}
\eeq
and
\beq
\Phi_t(z) \,\mid\, R^{at}  
=  \sqrt{t\,}\,e^{\pi i(t-1)/4} \, e^{-\pi i a't/12}
   \, \leg{a'}{t} \frac{\eta(z)}{\eta(z - a'/t)}
\qquad\mbox{(\cite[(25)]{At68b})}.
\mylabel{eq:PhitransRat}
\eeq
Although $E_2(z)$ is not a modular form, it well-known that
\beq
\mathcal{E}_{2,t}(z) := \frac{1}{t-1}\left(t\,E_2(tz) - E_2(z)\right),
\mylabel{eq:E2tdef}
\eeq
is an entire modular form of weight $2$ on $\Gamma_0(t)$ and
\beq
\mathcal{E}_{2,t}(z) \mid W_t = - \mathcal{E}_{2,t}(z).
\mylabel{eq:E2tWtrans}
\eeq
%% Let $K(z)$ be a modular function on $\Gamma_0(t)$, and let
%% %\bal
%% \begin{align}
%% \Phi_t(z) &= \Phi(z) = \frac{\eta(z)}{\eta(t^2z)}.
%% \mylabel{eq:Phidef},\\
%% K^{*}(z) &= K(z) \,\mid\, W_t = K\left(\frac{-1}{tz}\right).
%% \mylabel{eq:Ksdef}
%% \end{align}
%% %\eal
%% 
\begin{proposition}
\mylabel{propo:wt2AtkinLemma1}
Suppose $t\ge5$ is prime, $K(z)$ is a weakly holomorphic modular function
on $\Gamma_0(t)$, and 
\beq
S(z) = \mathcal{E}_{2,t}(tz) \, K^{*}(tz) \frac{\eta(z)}{\eta(t^2z)}
 - \chi_{12}(t)\,\eta(z)\,\sum_{n=m}^\infty \leg{1-24n}{t}\,\beta_t(n) 
                               q^{n-\frac{1}{24}},
\mylabel{eq:Sdef}
\eeq
where
\beq
\mathcal{E}_{2,t}(z) \, \frac{K(z)}{\eta(z)} 
   = \sum_{n=m}^\infty \beta_t(n) q^{n-\frac{1}{24}},
\mylabel{eq:betadef}
\eeq
and
\beq
K^{*}(z) = K(z) \stroke W_t.
\mylabel{eq:Ksdef}
\eeq
Then $S(z)$ is a weakly holomorphic modular form of weight $2$
on $\Gamma_0(t)$.
\end{proposition}
\begin{proof}
Suppose $t\ge5$ is prime and $K(z)$, $K^{*}(z)$, $S(z)$ are defined
as in the statement of the proposition. The function
\beq
H(z) := \Eis{t}{tz} \, \Phi_t(z) \, \Ks{tz}
\mylabel{eq:Hdef}
\eeq
is  a weakly holomorphic modular form of weight $2$ on $\Gamma_0(t^2)$.
As in \cite[Lemma1]{At68b} the function
\beq
S_1(z) := \sum_{a=0}^{t-1} H(z) \stroke R^{at}
\eeq
is  a weakly holomorphic modular form of weight $2$ on $\Gamma_0(t)$.
Utilizing \eqn{matrels1}, \eqn{PhitransRat}, 
\eqn{E2tWtrans} and the evaluation of a quadratic Gauss sum
\cite[(1.7)]{Be-Ev} we find that
\beq
S_1(z) = 
 \mathcal{E}_{2,t}(tz) \, K^{*}(tz) \frac{\eta(z)}{\eta(t^2z)}
   - \frac{1}{\sqrt{t}} \, e^{\pi i(t-1)/4} \, \eta(z) \,
     \sum_{a'=1}^{t-1} e^{-\pi ia't/12} \leg{a'}{t} \,
                       \Eis{t}{z-a'/t}\,
                       \frac{K(z - a'/t)}{\eta(z - a'/t)}
\mylabel{eq:S1id}
\eeq
Here we have also used the fact that
\beq
\Eis{t}{z} \stroke R^{at} = - \frac{1}{t}\, \Eis{t}{z - a'/t},
\mylabel{eq:E2Rattrans}
\eeq
where $a\,a'\equiv 1\pmod{t}$.
Hence 
$$
S_1(z) = S(z).
$$
This gives the result.
\end{proof}

We illustrate Proposition \propo{wt2AtkinLemma1} with two examples:
\begin{align}
 S(z) &= \Eis{5}{z} \, \left(\frac{\eta(z)}{\eta(5z)}\right)^6
\qquad\mbox{($K(z)=1$ and $t=5$)}
\mylabel{eq:Seg1}\\
\noalign{\mbox{and}}
 S(z) &= \Eis{7}{z} \, 
\left( \left(\frac{\eta(z)}{\eta(7z)}\right)^8
                       + 3\,\left(\frac{\eta(z)}{\eta(7z)}\right)^4 \right)
\qquad\mbox{($K(z)=1$ and $t=7$)}.
\mylabel{eq:Seg2}
\end{align}

\begin{corollary}
\mylabel{cor:SW}
Suppose $t\ge5$ is prime and  $S(z)$, $K(z)$ and the sequence $\beta_t(n)$
are defined
as in Proposition \propo{wt2AtkinLemma1}. Then
\beq
S(z) \stroke W_t 
=  - \eta(tz) \, \sum_{tn -s_t\ge m} \beta_t(tn - s_t)
                    q^{n-\frac{t}{24}}.
\mylabel{eq:SWz}
\eeq
\end{corollary}
\begin{proof}
The result follows easily from \eqn{matrels2}, \eqn{PhitransW}
and \eqn{E2tWtrans}.
\end{proof}

We illustrate the corollary by applying $W$ to both sides of the
equations \eqn{Seg1}--\eqn{Seg2}:
\begin{align}
 \sum_{n=1}^\infty \beta_5(5n-1) q^{n-\frac{5}{24}} &= 
5^3\,\frac{\Eis{5}{z}}{\eta(5z)} \, \left(\frac{\eta(5z)}{\eta(z)}\right)^6
\qquad\mbox{($K(z)=1$ and $t=5$)}
\mylabel{eq:SSeg1}\\
\noalign{\mbox{and}}
 \sum_{n=1}^\infty \beta_7(7n-2) q^{n-\frac{7}{24}} &= 
7^2\,\frac{\Eis{7}{z}}{\eta(7z)} 
\left( 3\,\left(\frac{\eta(7z)}{\eta(z)}\right)^4
                       + 7^2\,\left(\frac{\eta(7z)}{\eta(z)}\right)^8 \right)
\qquad\mbox{($K(z)=1$ and $t=7$)}.
\mylabel{eq:SSeg2}
\end{align}

For $t$ and $K(z)$ as in Proposition \propo{wt2AtkinLemma1}
we define
\begin{align}
\Psi_{t,K}(z) = \Psi_t(z) 
&=
\mathcal{E}_{2,t}(tz) \, \frac{\Ks{tz}}{\eta(t^2z)}
 - \chi_{12}(t)\,\sum_{n=m}^\infty \leg{1-24m}{t}\,\beta_t(n) 
                               q^{n-\frac{1}{24}}
\mylabel{eq:Psidef}\\
&\qquad  -  \sum_{t^2n -s_t\ge m} \beta_t(t^2n - s_t)
                    q^{n-\frac{1}{24}},
\nonumber
\end{align}
where $\Ks{z}$ and the sequence $\beta_t(n)$ is defined in 
\eqn{betadef}--\eqn{Ksdef}.
We have the following analogue of \propo{AtkinLemma1}.
\begin{corollary}
\mylabel{cor:SSerre}
The function $\Psi_{t,K}(z) \,  \eta(z)$ is a 
weakly holomorphic 
modular form of weight $2$ 
on the full modular group $\Gamma(1)$.
\end{corollary}
\begin{proof}
Let $S(z)$ be defined as in \eqn{Sdef},
so that $S(z)$ is a weakly holomorphic modular form of weight $2$
on $\Gamma_0(t)$. By \cite[Lemma 7]{Se}, the function
\beq
S(z) + S(z) \stroke W_t \stroke U
\eeq
is a modular form of weight $2$ on $\Gamma(1)$. Here $U=U_t$ is 
the Atkin operator
\beq
g(z) \stroke U_t = \frac{1}{t} \sum_{a=0}^{t-1} g\left(\frac{z+a}{t}\right).
\mylabel{eq:Udef}
\eeq
The result then follows from applying the $U$-operator to equation
\eqn{SWz}.
\end{proof}

We illustrate the $K(z)=1$ case of Corollary \cor{SSerre} with two examples:
\begin{align}
\Psi_5(z) & = \frac{E_4(z)^2 \, E_6(z)}{\eta(z)^{25}}
\mylabel{eq:Psi5}\\
\noalign{\mbox{and}}
\Psi_7(z) & = \frac{1}{\eta(z)^{49}}
\left( E_4(z)^5\, E_6(z) - 745 E_4(z)^2\, E_6(z)\,\Delta(z)\right).
\mylabel{eq:Psi7}
\end{align}

We need a weight $2$ analogue of \cite[Lemma 3]{At68b}.
For $t=5$, $7$ or $13$ the genus of $\Gamma_0(t)$ is zero,
and a Hauptmodul is
\beq
G_t(z) := \left(\frac{\eta(z)}{\eta(tz)}\right)^{24/(t-1)}.
\mylabel{eq:Hauptmodul}
\eeq
This function satisfies
\beq
G_t\left(\frac{-1}{tz}\right) = t^{12/(t-1)} G_t(z)^{-1}.
\mylabel{eq:Gttrans}
\eeq
\begin{proposition}
\mylabel{propo:wt2AtkinLemma3}
Suppose $t=5$, $7$ or $13$, and let $m$ be any negative integer such that
$24m\not\equiv1\pmod{t}$. Suppose  constants $k_j$ $(1\le j \le -m)$
are chosen so that
\beq
\beta_t(n) = 0, \qquad \mbox{for $m+1 \le n \le -1$},
\mylabel{eq:betatcond}
\eeq
where
\beq
K(z) = G_t(z)^{-m} + \sum_{k=1}^{-m-1} k_j G_t(z)^j
\mylabel{eq:KGm}
\eeq
and
\beq
\sum_{n=m}^\infty \beta_t(n) q^{n-\frac{1}{24}} = \Eis{t}{z} \frac{K(z)}{\eta(z)}.
\eeq
Then
\beq
\beta_t(n) = 0, \qquad \mbox{for}\quad \leg{1-24n}{t} = -\leg{1-24m}{t}.
\mylabel{eq:betatresult}
\eeq
\end{proposition}
\begin{proof}
Suppose $t=5$, $7$ or $13$, and $m$ is a negative integer such that
$24m\not\equiv1\pmod{t}$. Suppose $K(z)$ is chosen so that
\eqn{betatcond} holds. Let $S(z)$ be defined as in \eqn{Sdef},
and define
\beq
B(z) := S(z) + \chi_{12}(t) \, \leg{1-24m}{t}\, \Eis{t}{z}\,K(z),
\mylabel{eq:Bdef}
\eeq
so that
\beq
B(z) \stroke W_t = S^{*}(z) 
                 - \chi_{12}(t) \, \leg{1-24m}{t}\, \Eis{t}{z}\,\Ks{z},
\mylabel{eq:Btrans}
\eeq
where
\beq
S^{*}(z) = S(z) \stroke W_t.
\mylabel{eq:SSdef}
\eeq
Since $24m\not\equiv1\pmod{t}$, we see that
\beq
\ord_{0}(S(z)) = \ord_{i\infty}(S^{*}(z)) > 0.
\eeq
From \eqn{Gttrans} and \eqn{KGm} we see that
\beq
\ord_{0}(K(z)) = \ord_{i\infty}(K^{*}(z)) > 0
\eeq
and hence
\beq
\ord_{0}(B(z)) > 0.
\mylabel{eq:Bord0}
\eeq
Now
\beq
\ord_{i\infty}\left(
\mathcal{E}_{2,t}(tz) \, K^{*}(tz) \frac{\eta(z)}{\eta(t^2z)}
\right)\ge t - \frac{1}{24}(t^2-1) > 0,
\mylabel{eq:S1ord}
\eeq
for $t=5$, $7$, $13$. By construction the coefficient of $q^m$ in $B(z)$ is 
zero %%$0$
and so \eqn{betatcond}, \eqn{S1ord} imply that 
\beq
\ord_{i\infty}(B(z)) \ge 0.
\eeq
Therefore $B(z)$ is an entire modular form of weight $2$ and
hence a multiple of $\Eis{t}{z}$ since there are no nontrivial cusp
forms of weight $2$ on $\Gamma_0(t)$ for $t=5$, $7$ or $13$
by \cite{Co-Oe}. This implies that $B(z)$ is identically zero
by \eqn{Bord0}. Hence
\beq
\frac{B(z)}{E(q)}
 = q^{-s_t} \mathcal{E}_{2,t}(tz) \, K^{*}(tz) \frac{1}{E(q^{t^2})}
 - \chi_{12}(t)\,
\sum_{n=m}^\infty \left(\leg{1-24n}{t}-\leg{1-24m}{t}\right)\,\beta_t(n) 
                               q^{n}
=0.
\mylabel{eq:Bid}
\eeq
Since $-24s_t-1\equiv0\pmod{t}$, this implies that $\beta_t(n)=0$
whenever
$\leg{1-24n}{t} = -\leg{1-24m}{t}$.
\end{proof}

%%%%%%%%%%%%%%%%%%%%%%%%%%%%%%%%%%%%%%%%%%%%%%%%%%%%%%%%%%%%%%%%%%%%%%%%
%% Proposition 3.4 (ref=wt2AtkinLemma3)
%% > BETAKEG1:=series(EISP(5)/etaq(q,1,200)*(1/HM5^2+5/HM5),q,200):
%% > sift(BETAKEG1,q,5,1,199);
%%                                       0
%% > sift(BETAKEG1,q,5,2,199);
%%                                       0
%% > legendre(1-24*(-2),5);
%%                                       1
%% > seq([n,legendre(1-24*n,5)],n=0..4);
%%                   [0, 1], [1, -1], [2, -1], [3, 1], [4, 0]
%% > BETAKEG2:=series(EISP(7)/etaq(q,1,200)*(1/HM7),q,200):
%% > series(BETAKEG2,q,10);
%%           (-1)           2       5       6       7        9    / 12\
%%          q     + 1 - 15 q  + 49 q  - 24 q  + 88 q  - 311 q  + O\q  /
%% > legendre(1-24*(-1),7);
%%                                       1
%% > seq([n,legendre(1-24*n,7)],n=0..6);
%%           [0, 1], [1, -1], [2, 1], [3, -1], [4, -1], [5, 0], [6, 1]
%% > sift(BETAKEG2,q,7,1,199);
%%                                       0
%% > sift(BETAKEG2,q,7,3,199);
%%                                       0
%% > sift(BETAKEG2,q,7,4,199);
%%                                       0
%% This confirms some examples of Prop.3.4.
%%%%%%%%%%%%%%%%%%%%%%%%%%%%%%%%%%%%%%%%%%%%%%%%%%%%%%%%%%%%%%%%%%%%%%%%

We illustrate Proposition \propo{wt2AtkinLemma3} with two examples:
\begin{align}
&\sum_{n=-2}^\infty \beta_5(n) q^n = \frac{\Eis{5}{z}}{E(q)}\,\left(G_5(z)^2 + 5\, G_5(z)\right) 
\mylabel{eq:ALem3EG1}\\
&=q^{-2}+1-379\,q^3+625\,q^4+869\,q^5-20125\,q^8+23125\,q^9+25636\,q^{10}
-329236\,q^{13} +\cdots.
\nonumber
\end{align}
In this example, $t=5$ and $m=-2$, and we see that
$\beta_5(n)=0$ for $n\equiv1,2\pmod{5}$. In our second example, $t=7$ and
$m=-1$.
\begin{align}
&\sum_{n=-1}^\infty \beta_7(n) q^n = \frac{\Eis{7}{z}}{E(q)}\,G_7(z) 
\mylabel{eq:ALem3EG2}\\
&=q^{-1}+1-15\,q^2+49\,q^5-24\,q^6+88\,q^7-311\,q^9+392\,q^{12}-182\,q^{13}
+811\,q^{14}-1886\,q^{16} + \cdots.
\end{align}
In this example we see that
$\beta_7(n)=0$ for $n\equiv1,3,4\pmod{7}$.

%%HERE
%%
%%1/q-196884*q-42987520*q^2-2592899910*q^3-80983425024*q^4-1666013203000*q^5
%%-25512139800576*q^6-312598958503545*q^7+O(q^8)
The function
\begin{align}
&\frac{E_4(z)^2\,E_6(z)}{\Delta(z)} = \frac{E_6(z)}{E_4(z)}\,j(z)
\mylabel{eq:E46D} \\
& \quad = 
q^{-1} -196884\,q-42987520\,q^2-2592899910\,q^3-80983425024\,q^4-1666013203000\,q^5 + \cdots
\nonumber                    
\end{align}
is a modular form of weight $2$ on $\Gamma(1)$. As a modular form on
$\Gamma_0(t)$ it has a simple pole at $i\infty$ and a pole of order
$t$ at $z=0$. When $t=5$, $7$ or $13$, it is straightforward to show that
there are integers $a_{j,t}$ ($-1\le j \le t$) such that
\beq
\frac{E_6(z)}{E_4(z)}\,j(z) = \Eis{t}{z}\,\sum_{j=-1}^t a_{j,t} G_t(z)^j.
\mylabel{eq:E46Did}
\eeq
%%[1, 1], [2, 0], [3, -7*3^2*5^5], [4, -13*2^3*5^8], [5, -7*3^3*5^10], 
%%[6, -3*2^3*5^13], [7, -5^16]
For example,
\begin{align}
 \frac{E_6(z)}{E_4(z)}\,j(z)
&= \Eis{5}{z} \, \left( G_5(z)  
 - 3^2\cdot5^5\cdot7\,G_5(z)^{-1} 
 - \cdot2^3\cdot5^8\cdot13\,G_5(z)^{-2}  -
   \cdot3^3\cdot5^{10}\cdot7\,G_5(z)^{-3} \right.
\mylabel{eq:E46D5id}\\
&
\qquad \left. - 3\cdot2^3\cdot5^{13}\,G_5(z)^{-4}  
   -5^{16}\,G_5(z)^{-5}\right).
\nonumber
\end{align}

Reducing \eqn{E46Did} mod $t^c$ we obtain
a weight $2$ analogue of \cite[Lemma 4]{At68b}.
%% > coeff(SOL5[1],X[3],1);
%%                                     -196875
%% > modp(%,5^6);
%%                                      6250
%% > ifactor(%);
%%                                           5
%%                                    (2) (5) 
%%  
%% WEIGHT 2 ANALOG of ATKINs LEMMA 4
%% > series(E6*j/E4-EISP(5)/HM5+196875*EISP(5)*HM5,q,10);
%%                  2               3                4                  5
%%       -40625000 q  - 2576953125 q  - 80906250000 q  - 1665703125000 q 
%% 
%%                            6                    7    / 8\
%%          - 25511062500000 q  - 312595573828125 q  + O\q /
%% > modp(%,5^8);
%%                                      / 8\
%%                                     O\q /
%% > 5^8;
%%                                    390625
%% > modp(-196875,5^8);
%%                                    193750
%% > ifactor(%);
%%                                        5     
%%                                 (2) (5)  (31)
%% > ifactor(196875);
%%                                    2    5    
%%                                 (3)  (5)  (7)
%% > series(E6*j/E4-EISP(5)*(1/HM5 + 2*31*5^5*HM5),q,1000):
%% > modp(%,5^8);
%%                                     / 998\
%%                                    O\q   /
%% > series(E6*j/E4-EISP(7)/HM7,q,1000):
%% > modp(%,7^4);
%%                                     / 998\
%%                                    O\q   /
%% > series(E6*j/E4-EISP(13)/HM13,q,1000):
%% > modp(%,13^2);
%%                                     / 998\
%%                                    O\q   /
%% 
\begin{lemma}
\mylabel{lem:aLem4}
We have
\begin{align}
\frac{E_6(z)}{E_4(z)}\,j(z)
&\equiv \Eis{5}{z}\,\left(G_5(z) + 2\cdot31\cdot5^5\,G_5(z)^{-1}\right) \pmod{5^8},
\mylabel{eq:E46Dmod5}\\
\frac{E_6(z)}{E_4(z)}\,j(z) &\equiv  \Eis{7}{z}\,G_7(z) \pmod{7^4},
\mylabel{eq:E46Dmod7}\\
\frac{E_6(z)}{E_4(z)}\,j(z) &\equiv  \Eis{13}{z}\,G_{13}(z) \pmod{13^2}.
\mylabel{eq:E46Dmod13}
\end{align}
\end{lemma}

%% > series(j-1/HM7,q,1000):
%% > modp(%,7^4);
%%                                        / 998\
%%                                 748 + O\q   /
%% > series(j-1/HM13,q,1000):
%% > modp(%,13^2);
%%                                       / 998\
%%                                 70 + O\q   /
%% > a5:=750;
%%                                      750
%% > a51:=196875;
%%                                    196875
%% 
%% > series(j-1/HM5-a5 -a51*HM5,q,1000):
%% > modp(%,5^8);
%%                                     / 998\
%%                                    O\q   /
%% > ifactor(a51);
%%                                    2    5    
%%                                 (3)  (5)  (7)
%% > series(j-750-1/HM5 -63*5^5*HM5,q,1000):
%% > modp(%,5^8);
%%                                     / 998\
%%                                    O\q   /

We also need \cite[Lemma 4]{At68b}.
\begin{lemma}[Atkin \cite{At68b}]
\mylabel{lem:Lem4}[Atkin \cite{At68b}]
We have
\begin{align}
j(z)                                         
&\equiv G_5(z) + 750 + 3^2\cdot 7\cdot5^5\,G_5(z)^{-1} \pmod{5^8},
\mylabel{eq:jmod5}\\
j(z) &\equiv  G_7(z)  + 748 \pmod{7^4},
\mylabel{eq:jmod7}\\
j(z) &\equiv  G_{13}(z)  + 70\pmod{13^2}.
\mylabel{eq:jmod13}
\end{align}
\end{lemma}
\begin{remark}
In equation \eqn{jmod5} we have corrected a misprint in  \cite[Lemma 4]{At68b}.
\end{remark}

%% WEIGHT 2 ANALOG of ATKINs LEMMA 4 MOD 5^6
%% > modp(series(E6*j/E4-EISP(5)*(1/HM5 + 2*HM5*5^5),q,1000),5^6);
%%                                     / 998\
%%                                    O\q   /
%% > modp(series(E6*j^2/E4-EISP(5)*(750/HM5 + 1/HM5^2),q,1000),5^6);
%%                                     / 997\
%%                                    O\q   /
%% > modp(series(E6*j^3/E4-EISP(5)*(9375/HM5 + 1500/HM5^2+1/HM5^3),q,1000),5^6);
%%                                     / 996\
%%                                    O\q   /
%% 
To handle the $(t,c)=(5,6)$ case of Theorem \thm{aAtkincong} we will need
\begin{lemma}
\mylabel{lem:aLem4b}
\begin{align}
\frac{E_6(z)}{E_4(z)}\,j(z) 
&\equiv \Eis{5}{z}\left(G_5(z) + 2\cdot5^5\,G_5(z)^{-1}\right) \pmod{5^6},
\mylabel{eq:j1mod56}\\
\frac{E_6(z)}{E_4(z)}\,j(z)^2
&\equiv \Eis{5}{z}\left(2\cdot3\cdot5^3\,G_5(z) + G_5(z)^{2}\right) \pmod{5^6},
\mylabel{eq:j2mod56}\\
\noalign{\mbox{and}}%% \nonumber\\
\frac{E_6(z)}{E_4(z)}\,j(z)^a
&\equiv \Eis{5}{z}\left(\varepsilon_{1,a} \,G_5(z)^{a-2} + 
\varepsilon_{2,a} \, G_5(z)^{a-1} + G_5(z)^a\right) \pmod{5^6},
\mylabel{eq:jamod56}
\end{align}
for $a\ge3$, where $\varepsilon_{1,a}$,  $\varepsilon_{2,a}$ are integers satisfying
\beq
\varepsilon_{1,a}\equiv0\pmod{5^5}\quad \mbox{and}\qquad
\varepsilon_{2,a}\equiv0\pmod{5^3}.
\mylabel{eq:epmod56conds}
\eeq
\end{lemma}
\begin{proof}
The result can be proved from Lemmas \lem{aLem4} and \lem{Lem4},
some calculation and an easy induction argument.
\end{proof}

We need bases for $M_{2+12s_\ell}(\Gamma_0(t))$ for $t=5$, $7$, $13$.
The following result follows from \cite{Co-Oe} and by checking
the modular forms involved are
holomorphic at the cusps $i\infty$ and $0$.
\begin{lemma}
\mylabel{lem:bases}
Suppose $t=5$, $7$ or $13$ and $\ell>3$ is prime. Then
\beq
\dim M_{2+12s_\ell}(\Gamma_0(t)) = 1 + (1+t)\,s_\ell,
\mylabel{eq:dims}
\eeq
and the set
\beq
\left\{\Eis{t}{z}\,\Delta(z)^{s_\ell}\, G_t(z)^a\,:\, -ts_\ell \le a \le s_\ell
\right\}
\eeq
is a basis for $M_{2+12s_\ell}(\Gamma_0(t))$.
\end{lemma}

We are now ready to prove Theorem \thm{aAtkincong}.
We have two cases:
\subsection*{Case 1}
In the first case we assume that $(t,c)=(5,5)$, $(7,4)$ or $(13,2)$.
Suppose $\ell>3$ is prime and $\ell\ne t$.
By \eqn{Aellid} we have
\beq
\mathcal{A}_{\ell}(z) 
= \sum_{n=1}^{s_\ell} b_{n,\ell} \frac{E_6(z)}{E_4(z)} \,
                      \frac{j(z)^n}{\eta(z)}.
\mylabel{eq:Aellidv2}
\eeq
By Theorem \thm{Othm}, equation \eqn{ASXi} and Lemma \lem{bases} we have
\beq
\mathcal{A}_{\ell}(z) 
= \sum_{n=-ts_\ell}^{s_\ell}  d_{n,\ell} \frac{\Eis{t}{z}}{\eta(z)} \, 
                              G_t(z)^n,
\mylabel{eq:Aellid3}
\eeq
for some integers $d_{n,\ell}$ ($-t s_\ell \le n \le s_\ell$).
Now let
\beq
K(z) = \sum_{n=1}^{s_\ell} d_{n,\ell} G_t(z)^n.
\mylabel{eq:Kdef}
\eeq
By using Lemmas \lem{aLem4} and \lem{Lem4} to reduce equation \eqn{Aellidv2}
modulo $t^c$ and comparing the result with \eqn{Aellid3} we deduce that
\beq
\mathcal{A}_{\ell}(z) \equiv 
\mathcal{E}_{2,t}(z) \, \frac{K(z)}{\eta(z)} \pmod{t^c}
\mylabel{eq:Aellidcong}
\eeq
and that
\beq
d_{n,\ell} \equiv 0 \pmod{t^c},
\eeq
for $-t s_\ell \le n\le0$.
By examining \eqn{Aellid3} we see that
\beq
\mathcal{A}_{\ell}(z)  =              
\mathcal{E}_{2,t}(z) \, \frac{K(z)}{\eta(z)}  + O(q^{-\frac{1}{24}}).
\mylabel{eq:Aellidapprox}
\eeq
So if we let
\beq
\mathcal{E}_{2,t}(z) \, \frac{K(z)}{\eta(z)}
   = \sum_{n=-s_\ell}^\infty \beta_{t,\ell}(n) q^{n-\frac{1}{24}},
\mylabel{eq:betacofs},
\eeq
then \eqn{Aellidcong} may be rewritten as
\beq
\mathbf{a}(\ell^2 n - s_\ell)
+ \chi_{12}(\ell)  \left(\leg{1-24n}{\ell}-1-\ell\right)\mathbf{a}(n)
+  \ell\,\mathbf{a}\left( \frac{n + s_\ell}{\ell^2} \right)
\equiv \beta_{t,\ell}(n) \pmod{t^c}
\mylabel{eq:acongtc}
\eeq
and from \eqn{Aellidapprox} we have
\beq
\mathbf{a}(\ell^2 n - s_\ell)
+ \chi_{12}(\ell)  \left(\leg{1-24n}{\ell}-1-\ell\right)\mathbf{a}(n)
+  \ell\,\mathbf{a}\left( \frac{n + s_\ell}{\ell^2} \right)
=\beta_{t,\ell}(n)
\mylabel{eq:abetaid}
\eeq
for $-s_\ell \le n \le -1$.
Equation \eqn{abetaid} implies that
\begin{align}
\beta_{t,\ell}(-s_\ell) &= - \ell
\mylabel{eq:betares1}\\
\noalign{\mbox{and}}
\beta_{t,\ell}(n) &= 0,
\mylabel{eq:betares2}
\end{align}
for  $-s_\ell \le n \le -1$.
We can now apply Proposition \propo{wt2AtkinLemma3}
with $m=-s_\ell$ since $1-24m=\ell^2$ and $t\ne\ell$.
Hence
\beq
\beta_{t,\ell}(n) = 0,\quad\mbox{provided}\quad \leg{1-24n}{t}=-1.
\mylabel{eq:betares3}
\eeq
This gives Theorem \thm{aAtkincong}
when $(t,c)=(5,5)$, $(7,4)$ or $(13,2)$ by \eqn{acongtc}.

\subsection*{Case 2}
We consider the remaining case $(t,c)=(5,6)$ and assume $\ell > 5$ is
prime.
We proceed as in Case 1. This time when we use Lemma  \lem{aLem4b}
to reduce \eqn{Aellid} modulo $5^6$ we see that the only extra term occurs
when $n=1$. We find that with $K(z)$ as before we have
\beq
\mathcal{A}_{\ell}(z) \equiv 
\mathcal{E}_{2,t}(z) \, \frac{K(z)}{\eta(z)}  
+ b_{1,5} \cdot 2\cdot 5^5 \, 
\frac{\Eis{5}{z}}{\eta(z)} \, G_5(z)^{-1}
\pmod{5^6}.
\mylabel{eq:Aellidcong56}
\eeq
All that remains is to show that
\beq
b_{1,5} \equiv 0 \pmod{5},
\eeq
since then \eqn{Aellidcong} actually holds modulo $5^6$
and the rest of the proof proceeds as in Case 1.
Since $E_4(z) \equiv 1\pmod{5}$ we may reduce \eqn{Xiidcong1} modulo $5$
to obtain
\beq
\ell \, \Xi_\ell 
\equiv 
(\chi_{12}(\ell)\ell(1+\ell) - c_{0,\ell})\,\frac{E_6(z)}{E_4(z)} \, \frac{1}{\eta(z)}
- \sum_{n=1}^{s_\ell} (24n+1) c_{n,\ell} \, \frac{E_6(z)}{E_4(z)} \,
                             \frac{j(z)^n}{\eta(z)}
\pmod{5}.
\mylabel{eq:Ximod5}
\eeq
But
\beq
\mathcal{S}_{\ell}(z) \equiv 0 \pmod{5},
\eeq
by Theorem \thm{mainthm} (ii).
Hence
\beq
\ell \mathcal{A}_{\ell}(z)
\equiv 
(\chi_{12}(\ell)\ell(1+\ell) - c_{0,\ell})\,\frac{E_6(z)}{E_4(z)} \, \frac{1}{\eta(z)}
- \sum_{n=1}^{s_\ell} (24n+1) c_{n,\ell} \, \frac{E_6(z)}{E_4(z)} \,
                             \frac{j(z)^n}{\eta(z)}
\pmod{5}
\eeq
and we see that $b_{1,5}$ the coefficient of $ \frac{E_6(z)}{E_4(z)} \,
                             \frac{j(z)}{\eta(z)}$ is divisible by
$5$ as required. This completes the proof of Theorem \thm{aAtkincong}.

We close by illustrating Theorem \thm{aAtkincong} when $t=5$ and $\ell=7$.
In this case the theorem predicts that
$$
\mathbf{a}(49n-2) - \leg{1-24n}{7}\,\mathbf{a}(n) 
+ 7 \mathbf{a}\left(\frac{n+2}{49}\right)
\equiv - 8\,\mathbf{a}(n) \pmod{5^6},
$$
when $n\equiv 1,2\pmod{5}$. When $n=1$ this says
$$
149077845 \equiv -280 \pmod{5^6},
$$
which is easy to check.

%%%%%%%%%%%%%%%%%HERE

\noindent
\textbf{Acknowledgements}

\noindent
I would like to thank Ken Ono for sending me preprints
of his recent work \cite{On10}, \cite{On10b}.
%%Also \dots
%%
%%
\bibliographystyle{amsplain}

\end{document}